\documentclass[12pt]{amsart}
\usepackage{amscd}
\usepackage{amsmath}
\usepackage{stmaryrd}
\usepackage{amssymb}
\usepackage{units}
\usepackage[all]{xy}
\def\frk{\frak}               

\def\mm{{\frk m}}

\def\Phi{{\frk n}}
\def\Phi{{\frk N}}
\def\opn#1#2{\def#1{\operatorname{#2}}} 
\opn\chara{char} \opn\length{\ell} \opn\pd{pd} \opn\rk{rk}
\opn\projdim{proj\,dim} \opn\injdim{inj\,dim} \opn\rank{rank}
\opn\depth{depth} \opn\sdepth{sdepth} \opn\fdepth{fdepth}
\opn\grade{grade} \opn\height{height} \opn\embdim{emb\,dim}
\opn\codim{codim}  \opn\min{min} \opn\max{max}

\opn\Tr{Tr} \opn\bigrank{big\,rank}
\opn\superheight{superheight}\opn\lcm{lcm}
\opn\trdeg{tr\,deg}
\opn\reg{reg} \opn\lreg{lreg} \opn\ini{in} \opn\lpd{lpd}
\opn\size{size}
\opn\bigsize{bigsize}
\opn\div{div} \opn\Div{Div} \opn\cl{cl} \opn\Cl{Cl}
\opn\Spec{Spec} \opn\Supp{Supp} \opn\supp{supp} \opn\Sing{Sing}
\opn\Ass{Ass} \opn\Min{Min}
\opn\Ann{Ann} \opn\Rad{Rad} \opn\Soc{Soc}
\opn\Im{Im} \opn\Ker{Ker} \opn\Coker{Coker} \opn\Am{Am}
\opn\Hom{Hom} \opn\Tor{Tor} \opn\Ext{Ext} \opn\End{End}
\opn\Aut{Aut} \opn\id{id}  \opn\deg{deg}

\opn\nat{nat}
\opn\pff{pf}
\opn\Pf{Pf} \opn\GL{GL} \opn\SL{SL} \opn\mod{mod} \opn\ord{ord}
\opn\Gin{Gin} \opn\Hilb{Hilb}
\opn\aff{aff} \opn\con{conv} \opn\relint{relint} \opn\st{st}
\opn\lk{lk} \opn\cn{cn} \opn\core{core} \opn\vol{vol}
\opn\link{link} \opn\star{star}
\opn\gr{gr}

\def\pot#1#2{#1[\kern-0.28ex[#2]\kern-0.28ex]}

\opn\dirlim{\underrightarrow{\lim}}
\opn\inivlim{\underleftarrow{\lim}}
\let\Dirsum=\bigoplus

\let\to=\rightarrow

\def\Implies{\ifmmode\Longrightarrow \else
        \unskip${}\Longrightarrow{}$\ignorespaces\fi}
\def\implies{\ifmmode\Rightarrow \else
        \unskip${}\Rightarrow{}$\ignorespaces\fi}
\def\iff{\ifmmode\Longleftrightarrow \else
        \unskip${}\Longleftrightarrow{}$\ignorespaces\fi}

\let\:=\colon
\newtheorem{Theorem}{Theorem}[]
\newtheorem{Lemma}[Theorem]{Lemma}

\theoremstyle{definition}
\newtheorem{Remark}[Theorem]{Remark}

\newtheorem{Example}[Theorem]{Example}

\newtheorem{Definition}[Theorem]{Definition}

\let\epsilon\varepsilon
\let\phi=\varphi
\let\kappa=\varkappa
\def\qed{\ifhmode\textqed\fi
      \ifmmode\ifinner\quad\qedsymbol\else\dispqed\fi\fi}
\def\textqed{\unskip\nobreak\penalty50
       \hskip2em\hbox{}\nobreak\hfil\qedsymbol
       \parfillskip=0pt \finalhyphendemerits=0}
\def\dispqed{\rlap{\qquad\qedsymbol}}

\opn\dis{dis}
\def\pnt{{\raise0.5mm\hbox{\large\bf.}}}

\opn\Lex{Lex}

\begin{document}

\title{\bf Size and Stanley depth of monomial ideals}
\author{  Dorin Popescu}
\thanks{ We gratefully acknowledge the support from the project  ID-PCE-2011-1023, granted by the Romanian National Authority for Scientific Research, CNCS - UEFISCDI}

\address{Dorin Popescu, Simion Stoilow Institute of Mathematics of the Romanian Academy, Research unit 5,
University of Bucharest, P.O.Box 1-764, Bucharest 014700, Romania}
\email{dorin.popescu@imar.ro}

\maketitle

\begin{abstract}  The Lyubeznik size of a monomial ideal $I$ of a polynomial ring $S$ is a lower bound for the Stanley depth of $I$ decreased by $1$. A proof given by Herzog-Popescu-Vladoiu had a gap which is solved here.

\end{abstract}

\section*{Introduction}
Let  $S=K[x_1,\ldots,x_n]$, $n\in {\bf N}$, be a polynomial ring over a field $K$ and $\mm=(x_1.\ldots,x_n)$.
 Let $I\supsetneq J$  be two   monomial ideals of $S$ and  $u\in I \setminus J$ a monomial.
  For $Z\subset \{x_1,\ldots ,x_n\}$  with $(J:u)\cap K[Z]=0$, let $uK[Z]$ be the linear $K$-subspace of $I/J$ generated by the elements $uf$, $f\in K[Z]$.  A  presentation of $I/J$ as a finite direct sum of such spaces ${\mathcal D}:\
I/J=\Dirsum_{i=1}^ru_iK[Z_i]$ is called a {\em Stanley decomposition} of $I/J$. Set $\sdepth
(\mathcal{D}):=\min\{|Z_i|:i=1,\ldots,r\}$ and
\[
\sdepth\ I/J :=\max\{\sdepth \ ({\mathcal D}):\; {\mathcal D}\; \text{is a
Stanley decomposition of}\;  I/J \}.
\]

Let $h$ be the height of $a=\sum_{P\in \Ass_SS/I} P$ and $r$ the minimum $t$ such that there exist $\{P_1,\ldots,P_t\}\subset \Ass_S S/I$ such that $\sum_{i=1}^t P_i=a$. We call the {\em size} of $I$ the integer $\size_SI=n-h+r-1$. Lyubeznik \cite{L} showed that $\depth_SI\geq 1+\size_SI$.
If
Stanley's Conjecture \cite{S} would hold, that is  $\sdepth_S I/J\geq \depth_S I/J$, then we would get also $\sdepth_SI\geq 1+\size_SI$ as it is stated in \cite{HPV}. Unfortunately,  there exists a counterexample in \cite{DGKM} of this conjecture for $I=S$, $J\not =0$ and it is possible  that there are also counterexamples for $J=0$. However, the counterexample of \cite{DGKM} induces another one for  $J\not =0$ and $I\not =S$ generated by $5$ monomials, which shows that our result from \cite{P1} is tight. This counterexample does not affect Question 1 from \cite{P2}.

Y.-H. Shen noticed that the second statement of \cite[Lemma 3.2]{HPV} is false when $I$ is not squarefree and so the proof from \cite{HPV} of  $\sdepth_SI\geq 1+\size_SI$ is correct only when $I$ is squarefree. Since the depth is not a lower bound of sdepth due to \cite{DGKM} the lower bound of sdepth  given by size will have a certain value.
The main purpose of this paper is to show the above inequality in general (see Theorem \ref{t}).

 The important tool in the crucial point of the proof is the application of \cite[Theorem 4.5]{IKM} (a kind of polarization) to the so called the lcm-lattice associated to $I$ (see \cite{GPW}). Unfortunately, the polarization does not behaves well with size (see e.g. \cite[Example 1.2]{F}). Since it behaves somehow better with the so-called bigsize (very different from that introduced in \cite{P}, see Definition \ref{d}),  we have to  replace the size with the bigsize. Our bigsize is the right notion for a monomial squarefree ideal $I\subset S$ (see Theorem \ref{b}, an illustration of its proof is given in Examples \ref{e4}, \ref{e5}). If $I$ is not squarefree and $I^p\subset S^p$ is its polarization then it seems that a better notion will be $\bigsize_{S^p}(I^p)-\dim S^p+\dim S$.

The  inequality $\sdepth_SS/I\geq \size_SI$ conjectured in \cite{HPV} was proved in \cite{T} when $I$ is squarefree and it is  extended in \cite{F}. Our bigsize is useless for this inequality (see Remark \ref{r}). A similar inequality is proved  by Y.-H. Shen in the frame of the quotients of squarefree monomial ideals \cite[Theorem 3.6]{Sh}.

We owe thanks to Y.-H. Shen and S. A. Seyed Fakhari who noticed several mistakes in some  previous versions of this paper, and to B. Ichim, A. Zarojanu for a bad example.

\section{Squarefree monomial ideals}

The proof of the  the following theorem  is given in \cite{HPV} in a more general form, which is  correct only for squarefree ideals. For the sake of  completeness   we recall it here in sketch.
\begin{Theorem} (Herzog-Popescu-Vladoiu) \label{hpv} If $I$ is a squarefree monomial ideal then $$\sdepth_SI\geq \size_S(I)+1.$$
\end{Theorem}
\begin{proof} Write
$I=\cap_{i\in [s]}P_i$ as  an irredundant  intersection of  monomial prime ideals of $ S$ and assume that $P_1=(x_1,\ldots,x_r)$ for some $r\in [n]$. Apply induction on $s$, the case $s=1$ being trivial. Assume that  $s>1$. Using \cite[Lemma 3.6]{HVZ} we may reduce to the case when $\sum_{i\in [s]} P_i=\mm$.

Set  $S'=K[x_1,\ldots,x_r]$, $S''=K[x_{r+1},\ldots,x_n]$. For every nonempty proper subset $\tau\subset [s]$ set
$$S_{\tau}=K[\{x_i: i\in [r], x_i\not \in \sum_{j\in \tau} P_j\}],$$
$$J_{\tau}=(\cap_{i\in [s]\setminus\tau}P_i)\cap S_{\tau},\ \
L_{\tau}=(\cap_{i\in \tau}P_i)\cap S''.$$
If $J_{\tau}\not =0$, $L_{\tau}\not =0$ define $A_{\tau}=\sdepth_{S_{\tau}} J_{\tau}+\sdepth_{S''} L_{\tau}$. Also define $A_0=\sdepth_S I_0$ for $I_0=(I\cap S')S$. By  \cite[Theorem 1.6]{P} (the ideas come from \cite[Proposition 2.3]{AP}) we have
$$\sdepth_SI\geq \min\{A_0, \{A_{\tau}: \emptyset \not =\tau \subset [s], J_{\tau}\not =0, L_{\tau}\not =0\}\}.$$

 Using again \cite[Lemma 3.6]{HVZ} we see that if $I_0\not =0$ then $\sdepth_S I_0\geq n-r\geq \size_S(I)+1$. Fix a nonempty proper subset  $\tau\subset [s]$ such that
 $J_{\tau}\not =0, \ L_{\tau}\not =0$.
 It is enough to show that $A_{\tau}\geq \size_S(I)+1,$
that is to verify that
$\sdepth_{S''}L_{\tau}\geq \size_S(I)$
because $\sdepth_{S_{\tau}}(J_{\tau})\geq 1$.

 Set $P_{\tau}=\sum_{i\in \tau} P_i\cap S''$,  let us say $P_{\tau}=(x_{r+1},\ldots,x_e)$ for some $e\leq n$. Let $j_1<\ldots <j_t$ in $\tau$ with $t$ minim such that $\sum_{i=1}^tP_{j_i}\cap S''=P_{\tau}$. Thus $\size_{S''}L_{\tau}=t-1+n-e$.
Choose $k_1<\ldots<k_u$ in $[s]\setminus(\tau \cup \{1\})$ with $u$ minim such that
$(x_{e+1},\ldots,x_n)\subset \sum_{i=1}^uP_{k_i}$.  We have $u\leq n-e$. Then $P_1+\sum_{i=1}^tP_{j_i}+\sum_{i=1}^uP_{k_i}=\mm$ and so $u+t+1\geq \size_S(I)+1$. By induction hypothesis on $s$ we have $\sdepth_{S''}L_{\tau}\geq \size_{S''}L_{\tau}+1=t+n-e\geq t+u\geq \size_S(I).$
\hfill\ \end{proof}

Now let $I\subset S$ be a monomial ideal not necessarily squarefree and   $I=\cap_{i\in [s]}Q_i$ an irredundant   decomposition  of $I$  as an intersection of irreducible monomial ideals, $P_i=\sqrt{Q_i}$. Set $a=\sum_{i=1}^s P_i$.  Let $\nu$ be a total order on $ [s]$. We say that $\nu$ is {\em admissible} if given $i,j,k\in [s]$ with $j,k>i$ with respect to $\nu$  and such that from $\height (\sum_{p\in [i]}P_p+P_k)>\height (\sum_{p\in [i]}P_p+P_j)$ it follows that $j< k$.
  Let  ${\mathcal F}=(Q_{i_k})_{k\in [t]}$ be a family of ideals from $(Q_j)_{j\in [s]}$,  $t\in [s]$, $i_1<\ldots<i_t$ with respect to $\nu$ such that $P_{i_k}$ are maximal among $(P_i)_i$, and
 set $a_{k,{\mathcal F}}=\sum_{j=1}^{k} P_{i_j}\subset a$, $a_{0,{\mathcal F}}=0$,
  $a_{\mathcal F}=a_{t,{\mathcal F}}$, $t_{\mathcal F}=t$, $h_{\mathcal F}=\height a_{\mathcal F}$. Shortly, we speak about a family $\mathcal F$ of $I$. If $I$ is squarefree then each $P_j$ is maximal among  $(P_i)$.

 \begin{Definition} \label{d0}   A family ${\mathcal F}$ of $I$ with respect to $\nu$ is {\em admissible} if  $P_{i_k}\not\subset a_{k-1,{\mathcal F}}$ for all $k\in [t]$.  The admissible family  ${\mathcal F}$ is {\em maximal} if $a_{\mathcal F}=a$, that is, there exist no prime ideal $P\in \Ass_SS/I$ which  is not contained in $a_{\mathcal F}$.
 \end{Definition}

 \begin{Definition}\label{d}  Let $\mathcal F$ be a family of $I$ with respect to $\nu$. If $t_{\mathcal F}=1$ we set $\bigsize({\mathcal F})=\dim S/P_{i_1}$. If $t_{\mathcal F}>1$ then define by recurrence the $\bigsize({\mathcal F})=\min\{\bigsize({\mathcal F}'),1+\bigsize({\mathcal F}_1)\}$, where ${\mathcal F}'=(Q_{i_k})_{1\leq k< t}$ and ${\mathcal F}_1$ is the family obtained from the family $\widetilde{{\mathcal F}_1}=(Q_{i_t}+Q_{i_k})_{1\leq k< t}$ removing those ideals $ Q_{i_t}+Q_{i_k}$ which contain another ideal  $ Q_{i_t}+Q_{i_{k'}}$ with $k'\in [t-1]\setminus \{k\}$.
  Note that ${\mathcal F}_1$ is given by $\Ass_SS/I_1$, where $I_1=\cap_{1\leq k< t}(Q_{i_t}+Q_{i_k})$, the decomposition being not necessarily irredundant. Then  ${\mathcal F}_1$ is a family of $I_1$ with respect to the order induced by $\nu$ such that roughly speaking $Q_{i_t}+Q_{i_k}$ is smaller than
  $Q_{i_t}+Q_{i_{k'}}$ if $k<k'$ with respect to $\nu$.
  The integer $\bigsize({\mathcal F})$ is called the {\em bigsize} of $\mathcal F$. Note that  $\bigsize({\mathcal F})\leq t-1+\dim S/a_{\mathcal F}$. Set $\bigsize_{\nu}(I)=\bigsize({\mathcal F})$ for a maximal admissible family $\mathcal F$ of $I$ with respect to $\nu$. We call the {\em bigsize} of $I$
 the maximum $\bigsize_S(I) $ of $\bigsize_{\nu}(I)$ for all total admissible orders $\nu$ on $[s]$.
 \end{Definition}
 \begin{Remark} \label{r'} Note that given a total admissible  order $\nu$ there exists just one maximal admissible family  $\mathcal F$  with respect to $\nu$ so the above definition has sense.
 \end{Remark}
\begin{Example}\label{ex} Let $n=6$, $P_1=(x_1,x_2,x_4)$, $P_2=(x_1,x_3,x_4,x_6)$, $P_3=(x_2,x_3,x_4,x_6)$, $P_4=(x_1,x_4,x_5,x_6)$, $P_5=(x_1,x_2,x_3,x_5,x_6)$ and set $I=\cap_{i\in [5]}P_i$. Then ${\mathcal F}=\{P_1,P_2,P_5\}$,  ${\mathcal G}=\{P_1,P_3,P_4\}$ are maximal admissible  families of $I$ with respect of some total admissible order of $[5]$,  but $\bigsize({\mathcal F}')=\min\{3,1+1\}=2= \bigsize({\mathcal G}')$ and  $\bigsize({\mathcal F}_1)=0$, $\bigsize({\mathcal G}_1)=1$ which implies $\bigsize({\mathcal F})=1<2=\bigsize({\mathcal G})$. Note that $a_{k,{\mathcal F}}=a_{k,{\mathcal G}}$ for each $k\in [3]$.
\end{Example}

 \begin{Remark} \label{r''}
 Assume that $a_{\mathcal F}=(x_1,\ldots,x_r)$ for some $r\in [n]$. Set ${\tilde S}=K[x_1,\ldots,x_r]$ and let ${\tilde{\mathcal F}}=(Q_{i_k}\cap {\tilde S})_{k\in [t]}$. Then $\bigsize({\mathcal F})=n-r+\bigsize({\tilde {\mathcal F}})$.
 \end{Remark}
 \begin{Remark} \label{r0}
 Let ${\mathcal F}=(Q_{i_k})_{k\in [t]}$ be a an admissible family of $I$ with respect to a total admissible order $\nu$ and $r\in [t-1]$. Then ${\mathcal G}=(Q_{i_k})_{k\in [r]}$ is an admissible family of $I$ with respect to  $\nu$ and $\bigsize({\mathcal F})\leq \bigsize({\mathcal G})$.
 \end{Remark}
 \begin{Remark} \label{r4}
 Let ${\mathcal F}=(Q_{i_k})_{k\in [t]}$ be a a  family of $I$ with respect to a total admissible order $\nu$.   Then  $\bigsize({\mathcal F})=r-1+\dim S/(P_{i_{k_1}}+\ldots +P_{i_{k_r}})$ for some $k_1<\ldots <k_r$ from $[t]$.
 \end{Remark}

 \begin{Example} \label{e} Let $n=5$, $P_1=(x_1,x_2)$, $P_2=(x_2,x_3)$, $P_3=(x_1,x_4,x_5)$ and $I=P_1\cap P_2\cap P_3$. Then  ${\mathcal F}=(P_i)_{i\in [3]}$ is a maximal admissible family of $I$ with respect to the usual order $\nu$ and  $\size_SI=1$ because $P_2+P_3=\mm$. Note that ${\mathcal F}'=(P_i)_{i=1,2}$ has $\bigsize({\mathcal F}')=\min\{3,1+2\}=3$ and ${\mathcal F}_1=(P_3+P_i)_{i=1,2}$ has $\bigsize({\mathcal F}_1)=1$. Thus $\bigsize({\mathcal F})=2$.

  The order given by  $I=P_2\cap P_3\cap P_1$ is not admissible, but the order $\nu'$ given by $I=P_2\cap P_1\cap P_3$ is admissible. The family ${\mathcal G}=(P_i)_{i=2,1,3}$ has $\bigsize({\mathcal G}')=\min\{3,1+2\}=3$ and ${\mathcal G}_1=(P_3+P_i)_{i=2,1}$ has $\bigsize({\mathcal G}_1)=1$. Thus $\bigsize({\mathcal G})=\min\{3,1+1\}=2$. Similarly, the order $\nu''$ given by  $\{3,1,2\}$ is total admissible, the family ${\mathcal H}=(P_i)_{i=3,1,2}$ has $\bigsize({\mathcal H}')=\min\{2,1+1\}=2$ and ${\mathcal H}_1=(P_2+P_i)_{i=3,1}$ has $\bigsize({\mathcal H}_1)=1$. Thus $\bigsize({\mathcal H})=\min\{2,1+1\}=2$ and
 we have $\bigsize_{\nu'',S}(I)=2$.
\end{Example}

 \begin{Example} \label{e1} Let $n=2$, $Q_1=(x_1)$, $Q_2= (x_1^2,x_2)$  and $I=Q_1\cap Q_2$. Then $P_2$ is the only prime $P_i$ maximal among $(P_j)_{j\in [2]}$ and for ${\mathcal F}=\{P_2\}$ we have  $\bigsize_S({\mathcal F})=\size_S(I)=0$.
 \end{Example}
 \begin{Example} \label{e0} Let $n=4$, $Q_1=(x_1,x_2^2)$, $Q_2= (x_2,x_3)$, $Q_3=(x_3^2,x_4)$  and $I=Q_1\cap Q_2\cap Q_3$. Then  ${\mathcal F}=(Q_i)_{i\in [3]}$ is a maximal admissible family of $I$ with respect to the usual order $\nu$ and  $\size_SI=1$ because $P_1+P_3=\mm$. Note that ${\mathcal F}'=(Q_i)_{i=1,2}$ has $\bigsize({\mathcal F}')=\min\{2,1+1\}=2$ and ${\mathcal F}_1=(Q_3+Q_i)_{i=1,2}$ has $\bigsize({\mathcal F}_1)=0$. Thus $\bigsize({\mathcal F})=\min\{2,1+0\}=1$.

 The order  $\nu'$ given by $I=P_2\cap P_1\cap P_3$ is admissible. The family ${\mathcal G}=(Q_i)_{i=2,1,3}$ has $\bigsize({\mathcal G}')=2$ and ${\mathcal G}_1=(Q_3+Q_i)_{i=2,1}$ has $\bigsize({\mathcal G}_1)=0$. Thus $\bigsize({\mathcal G})=\min\{2,1+0\}=1$ and
 we have $\bigsize_{\nu',S}(I)=1$. Similarly, the order  $\nu''$ given by $\{2,3,1\}$ is total admissible and $\bigsize_{\nu''}(I)=1$. Also note that the order ${\bar \nu}$ given by $\{3,2,1\}$ is total admissible, the family ${\mathcal H}=(Q_i)_{i=3,2,1}$ has $\bigsize({\mathcal H}')=\min\{2,1+1\}=2$ and ${\mathcal H}_1=(Q_1+Q_i)_{i=3,2}$ has $\bigsize({\mathcal H}_1)=0$. Thus $\bigsize({\mathcal H})=\min\{2,1+0\}=1$ and
 we have $\bigsize_{{\bar\nu},S}(I)=1$.
 \end{Example}

\begin{Example} \label{e'} Let $n=6$, $P_1=(x_1,x_2)$, $P_2= (x_1,x_3)$, $P_3=(x_1,x_6)$, $P_4=(x_3,x_4)$, $P_5=(x_3,x_5)$, $P_6=(x_2,x_4)$, $P_7=(x_5,x_6)$ and $I=\cap_{i\in [7]} P_i$. Let  $\nu$ be the usual order and ${\mathcal F}=(P_i)_{i\in [5]}$. Then $\mathcal F$ is maximal admissible and  $\bigsize({\mathcal F})=4>\size_SI$.  Taking $\nu'$ given by the order $\{7,5,3,1,4\}$ we get a maximal admissible family $\mathcal G$ with $\bigsize({\mathcal G})=3$. Thus $\bigsize_S(I)=4>3=\size_SI$.
 \end{Example}

\begin{Lemma}\label{l0}  Let  $\nu$ be a total admissible order on $[s]$ and ${\mathcal F}=(Q_{i_k})_{k\in [t]} $ a family of $I$ with respect to $\nu$. Then
$ \bigsize(\mathcal F)\geq \size_SI.$
\end{Lemma}
\begin{proof}
  By Remark \ref{r4}  we have  $\bigsize({\mathcal F})=r-1+\dim S/(P_{i_{k_1}}+\ldots +P_{i_{k_r}})$ for some $k_1<\ldots <k_r$ from $[t]$.
  We may suppose that $\sum_{j\in [r]}P_{i_{k_j}}=(x_1,\ldots,x_e)$ for some $e\in [n]$. Choose for each $p>e$, $p\leq n$ an $u_p\in [s]$ such that $x_p\in P_{u_p}$. Then $\sum_{j\in [r]}P_{i_{k_j}}+\sum_{p=e+1}^n P_{u_p}=\mm$ and so $\size I\leq  r-1+\dim S/(P_{i_{k_1}}+\ldots +P_{i_{k_r}})=\bigsize({\mathcal F})$.
\hfill\ \end{proof}

Next we present a slightly extension of Theorem \ref{hpv}.

\begin{Theorem}\label{b}  Let  $I=\cap_{i\in [s]}P_i$  be an irredundant  intersection of  monomial prime ideals of $ S$.  Then $\sdepth_SI\geq 1+ \bigsize_S(I)$.
\end{Theorem}
\begin{proof}  Using \cite[Lemma 3.6]{HVZ} we may reduce to the case when $\sum_{i\in [s]} P_i=\mm$. Apply induction on $n$.  Assume  that  $\bigsize_S(I) =\bigsize(\mathcal F)$ for a maximal admissible family ${\mathcal F}=(P_{i_k})_{k\in [t]}$ of $I$ with respect to a total admissible order $\nu$.  We may suppose that $i_t=s$ and  $\sum_{k\in [t-1]}P_{i_k}=(x_{r+1},\ldots,x_n)$, $r\geq 1$. Set $S'=K[x_1,\ldots,x_r]$, $S''=K[x_{r+1},\ldots,x_n]$.  We may use  \cite[Theorem 1.6]{P} even when $(x_1,\ldots,x_r)\not \in \Ass_SS/I$ (see \cite[Lemma 2.1]{HPV}).
In the notations of Theorem \ref{hpv} we have
$$\sdepth_SI\geq \min\{A_0, \{A_{\tau}: \emptyset \not =\tau \subset [s], J_{\tau}\not =0, L_{\tau}\not =0\}\}.$$
If $I_0=(I\cap S')S\not =0$ then $A_0\geq  1+(n-r)\geq 1+\dim S/P_{i_t}\geq 2+\bigsize({\mathcal F}_1)\geq  1+\bigsize_S(I)  $.

 Now suppose that $\sdepth_SI\geq A_{\tau}$ for some $\tau\subset [s]$ with $J_{\tau}\not =0$, $ L_{\tau}\not =0$. Thus $i_k$ must be in $\tau$ for any $k\in [t-1]$ because otherwise $J_{\tau}=0$. Then  ${\mathcal H}=(P_{i_k}\cap S'')_{k\in [t-1]}$ is a maximal admissible family of $L_{\tau}$ with respect to $\nu$. Note that $\bigsize(\mathcal H)\geq \bigsize({\mathcal F}_1)$.  By induction hypothesis on $n$ we have $$\sdepth_{S''}L_{\tau}\geq 1+\bigsize_{S''}(L_{\tau})\geq 1+\bigsize(\mathcal H)\geq 1+\bigsize({\mathcal F}_1)\geq \bigsize({\mathcal F}).$$
Therefore,
$$\sdepth_SI\geq A_{\tau}\geq 1+\sdepth_{S''}L_{\tau}\geq 1+\bigsize({\mathcal F})= 1+ \bigsize_S(I).$$
\hfill\ \end{proof}
\begin{Example}\label{e4}  We illustrate the above proof on the case of $\mathcal F$ given in Example \ref{e'}. Set $S'=K[x_5]$, $S''=K[x_1,x_2,x_3,x_4,x_6]$. Then $\tau=\{1,2,3,4,6\}$ is the only $\tau\subset [7]$ such that $J_{\tau}\not = 0$. We have $\sdepth_SI=5=1+\sdepth_{S''}L_{\tau}$.
Set ${\tilde S}'=K[x_4]$, ${\tilde S}''=K[x_1,x_2,x_3,x_6]$. Then ${\tilde \tau}=\{1,2,3\}$ is the only ${\tilde \tau}\subset \tau=[7]\setminus \{5,7\}$ such that $J_{\tilde \tau}\not =0$.
We have $\sdepth_{ S''}L_{\tau}=4=1+\sdepth_{{\tilde S}''}L_{\tilde \tau}$.
Now set ${\hat S}'=K[x_6]$, ${\hat S}''=K[x_1,x_2,x_3]$. Then ${\hat \tau}=\{1,2\}$ is the only ${\hat \tau}\subset {\tilde \tau}=[7]\setminus \{4,5,6,7\}$ such that $J_{\hat \tau}\not =0$. We have $\sdepth_{{\tilde S}''}L_{\tilde \tau}=3=1+\sdepth_{{\hat S}''}L_{\hat \tau}$.

On the other hand, ${\mathcal H}=\{P_1\cap S'', P_2\cap S'',P_3\cap S'',P_4\cap S''\}$ is a maximal admissible family of $L_{\tau}$ and we have $\bigsize({\mathcal H})=3=\bigsize({\mathcal F}_1)$. Also note that ${\mathcal P}=\{P_1\cap {\tilde S}'',P_2\cap {\tilde S}'',P_3\cap {\tilde S}''\}$ is a maximal admissible family of $L_{\tilde \tau}$ and $\bigsize({\mathcal P})=2=\bigsize({\mathcal H}_1)$. Finally, ${\mathcal E}=\{P_1\cap {\hat S}'',P_2\cap {\hat S}''\}$ is a maximal admissible family of $L_{\hat \tau}$ and $\bigsize({\mathcal E})=1=\bigsize({\mathcal P}_1)$.
Therefore, we have  $\sdepth_SI=1+\bigsize({\mathcal F})$,
 $\sdepth_{S''}L_{\tau}=1+\bigsize({\mathcal H})$,   $\sdepth_{{\tilde S}''}L_{\tilde \tau}=1+\bigsize({\mathcal P})$ and  $\sdepth_{{\hat S}''}L_{\hat \tau}=1+\bigsize({\mathcal E})$.
\end{Example}
\begin{Remark}\label{r} Note that in Example \ref{e4} we have $\sdepth_SS/I=3=\bigsize({\mathcal G})<4=\bigsize({\mathcal F})=\bigsize_S(I)$ which shows that the corresponding inequality for $S/I $ fails using this bigsize.
 As $\sdepth_{S''}S''/L_{\tau}=3$ too, we see that the proof of Theorem \ref{b} fails in the case of the module $S/I$. Thus the so-called the splitting of variables for arbitrary $r$ from \cite[Proposition 2.1]{HPV} does not hold for $S/I$ (this holds for the case when $r$ is given by a so-called main prime as it is used in \cite{T}).
\end{Remark}

\begin{Example}\label{e5}  We consider now the case of $\mathcal G$ given in Example \ref{e'}. Set $S'=K[x_4]$, $S''=K[x_1,x_2,x_3,x_5,x_6]$. Then $\tau=\{1,2,3,5,7\}$ is the only $\tau\subset [7]$ such that $J_{\tau}\not = 0$. We have $\sdepth_SI=5=1+\sdepth_{S''}L_{\tau}$.
Set ${\tilde S}'=K[x_2]$, ${\tilde S}''=K[x_1,x_3,x_5,x_6]$. Then ${\tilde \tau}=\{3,5,7\}$ is the only ${\tilde \tau}\subset \tau= [7]\setminus \{4,6\}$ such that $J_{\tilde \tau}\not =0$.
We have $\sdepth_{ S''}L_{\tau}=4=1+\sdepth_{{\tilde S}''}L_{\tilde \tau}$.
Now set ${\hat S}'=K[x_1]$, ${\hat S}''=K[x_3,x_5,x_6]$. Then ${\hat \tau}=\{5,7\}$ is the only ${\hat \tau}\subset{\tilde \tau}= [7]\setminus \{2,4,6\}$ such that $J_{\hat \tau}\not =0$. We have $\sdepth_{{\tilde S}''}L_{\tilde \tau}=3=1+\sdepth_{{\hat S}''}L_{\hat \tau}$.

On the other hand, ${\mathcal H}=\{P_7\cap S'', P_5\cap S'',P_3\cap S'',P_1\cap S''\}$ is a maximal admissible family of $L_{\tau}$ and we have $\bigsize({\mathcal H})=2=\bigsize({\mathcal G}_1)$. Also note that ${\mathcal P}=\{P_7\cap {\tilde S}'',P_5\cap {\tilde S}'',P_3\cap {\tilde S}''\}$ is a maximal admissible family of $L_{\tilde \tau}$ and $\bigsize({\mathcal P})=2>1=\bigsize({\mathcal H}_1)$. Finally, ${\mathcal E}=\{P_7\cap {\hat S}'',P_5\cap {\hat S}''\}$ is a maximal admissible family of $L_{\hat \tau}$ and $\bigsize({\mathcal E})=1=\bigsize({\mathcal P}_1)$. Therefore, we have  $\sdepth_SI>1+\bigsize({\mathcal G})$,
 $\sdepth_{S''}L_{\tau}>1+\bigsize({\mathcal H})$, $\sdepth_{{\tilde S}''}L_{\tilde \tau}=1+\bigsize({\mathcal P})$ and  $\sdepth_{{\hat S}''}L_{\hat \tau}=1+\bigsize({\mathcal E})$.
\end{Example}

\section{Bigsize and Stanley depth}

 Let $I\subset S$ be a monomial ideal and  $I=\cap_{i\in [s]}Q_i$ an irredundant   decomposition  of $I$  as an  intersection of irreducible monomial ideals, $P_i=\sqrt{Q_i}$. Let $G(I) $ be the minimal set of monomial generators of $I$. Assume that  $\sum_{P\in \Ass_S S/I} P=\mm$.  Given $j\in [n]$ let $\deg_j I$ be the  maximum degree of $x_j$ in all monomials of  $G(I)$.

\begin{Lemma}\label{l}
Suppose that   $c:=\deg_nI>1$,  let us say $c=\deg_nQ_j$ if and only if  $j\in [e]$ for some $e\in [s]$.
 Assume that  $Q_j=(J_j,x_n^c)$ for some irreducible ideal $J_j\subset S_n=K[x_1,\ldots,x_{n-1}]$, $j\in [e]$. Let $Q'_j=(J_j,x_n^{c-1})\subset S$, $Q''_j=(J_j,x_{n+1})\subset {\tilde S}=S[x_{n+1}]$ and set
 $${\tilde I}=(\cap_{i=e+1}^sQ_i{\tilde S})\cap (\cap_{i\in [e]}Q'_i{\tilde S})\cap (\cap_{i=s+1}^ {s+e}Q_i)\subset {\tilde S},$$
  where $Q_i=Q''_{i-s}$ for $i>s$, the decomposition of ${\tilde I}$ being not necessarily irredundant. Then $\sdepth_{{\tilde S}}{\tilde I}\leq \sdepth_SI+1$ and $\sdepth_{\tilde S}{\tilde S}/{\tilde I}\leq \sdepth_SS/I+1$.
  \end{Lemma}
\begin{proof}  Let $L_I$, $L_{\tilde I}$ be the LCM-lattices associated to $I,{\tilde I}$. The map ${\tilde S}\to S$ given by $x_{n+1}\to x_n$ induces a surjective join-preserving map
$L_{\tilde I}\to L_I$ and by \cite[Theorem 4.5]{IKM}  we get $\sdepth_{\tilde S}{\tilde I}\leq \sdepth_SI+1$ and  $\sdepth_{{\tilde S}}{\tilde S}/{\tilde I}\leq \sdepth_SS/I+1$.
 \hfill\ \end{proof}

 With the notations and assumptions  of Lemma \ref{l} let
$${\tilde C}=\{i\in [s]:P_i{\tilde S}\in \Ass_{\tilde S}{\tilde S}/{\tilde I}\}\cup ([s+e]\setminus [s]).$$
 Choose
 a total admissible order $\tilde \nu$ on ${\tilde C}$  and  a total admissible order $\nu$ on $[s]$ extending the restriction  of ${\tilde \nu}$ to $[s]\cap {\tilde C}$. Let $\tilde{\mathcal F}=({\tilde Q}_{i_k})_{k\in [t]}$ be a  family of $\tilde I$ with respect to $\tilde \nu$. Replace in $\tilde{\mathcal F}$ the ideals ${\tilde Q}_{i_k}$ by $Q_{i_k}={\tilde Q}_{i_k}\cap S$ when $P_{i_k}$ is maximal in $\Ass_SS/I$ and $\tilde{Q}_{i_k}$ is not of the form $Q'_i{\tilde S}$ or $Q''_i$ for some $i\in [e]$. When ${\tilde Q}_{i_k}$ is of the form
$Q'_i{\tilde S}$ or $Q''_i$ for some $i\in [e]$ then replace in $\tilde{\mathcal F}$ the ideal  ${\tilde Q}_{i_k}$ by $Q_i$. If $P_{i_k}$ is not maximal in $\Ass_SS/I$ then ${\tilde Q}_{i_k}\subset Q'_i{\tilde S}$ for some $i\in [e]$ and we replace  in $\tilde{\mathcal F}$ the ideal ${\tilde Q}_{i_k}$ by $Q_i$ (this $i$ is not unique and we have to choose a possible one). Note that $x_n\in P_{i_k}$ because otherwise $Q_{i_k}\subset Q_i$ which is impossible.

In this way, we get a family $\overline{\mathcal F}$ of ideals which are maximal in $\Ass_SS/I$. Sometimes $\overline{\mathcal F}$ contains  the same ideal $Q_i$, $i\in [e]$ several times.  Keeping such $Q_i$ in  $\overline{\mathcal F}$ only the first time when it appears and removing the others we get a family ${\mathcal F}$ of $I$ with respect to $\nu$.

\begin{Lemma}\label{l4} There exists a family $\mathcal G$ of $I$ with respect to $\nu$ such that   $\bigsize({\tilde{\mathcal F}})\geq 1+\bigsize({\mathcal G})$.
  \end{Lemma}
\begin{proof}
Apply induction on $t$. Assume that $t=1$. Then note that $\bigsize({\tilde{\mathcal F}})=\dim {\tilde S}/{\tilde P}_{i_1}=1+\dim  S/ P_{i_1}=1+\bigsize({\mathcal F})$
 when $P_{i_1}$ is maximal in $\Ass_SS/I$ and $\tilde{Q}_{i_1}$ is not of the form $Q'_i{\tilde S}$ or $Q''_i$ for some $i\in [e]$. If $\tilde{Q}_{i_1}=Q'_i{\tilde S}$ for some $i\in [e]$ then $\bigsize({\tilde{\mathcal F}})=\dim {\tilde S}/ P_i{\tilde S}=1+\dim  S/ P_i=1+\bigsize({\mathcal F})$. Similarly, it happens when  $\tilde{Q}_{i_1}=Q''_i$ because then $\dim {\tilde S}/{\tilde P}_{e+i}=\dim S/J_i=1+\dim S/P_i$.  If  $\tilde{Q}_{i_1}=Q_l{\tilde S}$ for some $l\in [s]$ such that $Q_l\subset Q'_i$ for some $i\in [e]$ and $P_l$ is not maximal in $\Ass_SS/I$ then  note that $\dim {\tilde S}/P_l{\tilde S}=1+\dim S/P_l>1+\dim S/P_i$.

 Let $t>1$. Assume that  $\bigsize(\tilde { \mathcal F})=t-1+\dim {\tilde S}/a_{\tilde{\mathcal F}}$. As above we see that $\dim {\tilde S}/a_{\tilde{\mathcal F}}\geq 1+\dim S/a_{\mathcal F}$.
Let $\overline{\mathcal F}=(Q_{i_k})_{k\in [t]}$. If $\overline{\mathcal F}=\mathcal F$ then we get $\bigsize( { \mathcal F})\leq t-1+\dim S/a_{\mathcal F}\leq \bigsize(\tilde { \mathcal F})-1$. Otherwise,
assume that ${\mathcal F}=(Q_{i_k})_{k\in E}$ for some $E\subsetneq [t]$. We have  $\bigsize( { \mathcal F})\leq |E|-1+\dim S/a_{\mathcal F}<t-1+\dim S/a_{\mathcal F}\leq \bigsize(\tilde { \mathcal F})-1$. Then take ${\mathcal G}={\mathcal F}$.

Now assume that  $\bigsize(\tilde{\mathcal F})=r-1+\dim {\tilde S}/\sum_{j\in [r]} {\tilde P}_{i_{k_j}}$ for some $r\in [t-1]$ and $k_1<\ldots<k_r$ from $[t]$ (see Remark \ref{r4}). Set $\tilde{\mathcal G}=({\tilde Q}_{i_{k_j}})_{j\in [r]}$. We have $\bigsize(\tilde{\mathcal G})\leq r-1+\dim {\tilde S}/a_{\tilde{\mathcal G}}=\bigsize(\tilde{\mathcal F})$. Consider the families $\overline{\mathcal G}$, ${\mathcal G}$ corresponding to $\tilde{\mathcal G}$ similarly to  $\overline{\mathcal F}$, ${\mathcal F}$ corresponding to $\tilde{\mathcal F}$. By induction hypothesis ($r<t$) we have $\bigsize( \tilde{\mathcal G})\geq 1+\bigsize({\mathcal G})$.  Then
 $$\bigsize(\tilde{\mathcal F})\geq \bigsize(\tilde{\mathcal G})\geq 1+\bigsize({\mathcal G}).$$
    \hfill\ \end{proof}
 \begin{Example} \label{e2}
 Let $n=4$, $Q_1=(x_1,x_2)$, $Q_2=(x_1,x_3)$, $Q_3=(x_1^2,x_2,x_3)$, $Q_4=(x_1^2,x_3,x_4)$ and $I=\cap_{i\in [4]}Q_i$. Let ${\mathcal F}=\{Q_3,Q_4\}$. Then $\size_S(I)=1$ because $P_3+P_4=\mm$. Also note that $\bigsize({\mathcal F}')= \min\{1,1+0\}=1$, $\bigsize({\mathcal F}_1)=0$  and so $\bigsize({\mathcal F})= \min\{1,1+0\}=1$.

  Clearly,
  ${\tilde I}=Q_1{\tilde S}\cap Q_2{\tilde S}\cap Q''_3\cap Q''_4$. Now $P_1{\tilde S},P_2{\tilde S}$ are maximal in $\Ass_{\tilde S}{\tilde S}/{\tilde I}$. For ${\mathcal G}=\{Q_1{\tilde S},Q_2{\tilde S},Q''_3,Q''_4\}$ we get $\bigsize({\mathcal G}')=\min\{\min\{3,1+2\}, 1+1\}=2$, $\bigsize({\mathcal G}_1)=\min\{1,1+0\}=1$ and so    $\bigsize({\mathcal G})=\min\{2,1+1\}=2$.  If we take ${\mathcal H}=\{Q''_3,Q''_4,Q_1\}$ then   $\bigsize({\mathcal H}')=\min\{2,1+1\}=2$, $\bigsize({\mathcal H}_1)=1$
    and so  $\bigsize({\mathcal H})=2$. Thus
   $\bigsize({\mathcal G})=\bigsize({\mathcal H})=1+\bigsize({\mathcal F})$.
 \end{Example}

 \begin{Example} \label{e3} Let $n=4$, $Q_1=(x_1,x_2)$, $Q_2=(x_1^2,x_3)$, $Q_3=(x_1^2,x_4)$ and $I=\cap_{i\in [3]}Q_i$. Let ${\mathcal F}=\{Q_1,Q_2,Q_3\}$. Then we see that $\bigsize({\mathcal F})=2=\size I$.
 Clearly,  ${\tilde I}=Q_1{\tilde S}\cap Q'_2{\tilde S}\cap Q'_3{\tilde S}\cap Q''_2\cap Q''_3$, where
$Q'_2= (x_1,x_3)$, $Q'_3= (x_1,x_4)$, $Q''_2= (x_3,x_5)$, $Q''_3= (x_4,x_5)$. Then $\{Q_1{\tilde S},Q''_2,Q''_3\}$, $\{Q'_2{\tilde S},Q_1{\tilde S},Q''_3\}$, $\{Q'_3{\tilde S},Q_1{\tilde S},Q''_2\}$,
$\{Q''_2, Q_1{\tilde S},Q''_3\}$, $\{Q''_3,Q_1{\tilde S},Q'_2{\tilde S}\}$ are maximal admissible families of ${\tilde I}$ but with respect to some total orders which are not admissible.   However, $\mathcal G=\{Q''_2,Q'_2{\tilde S},Q_1{\tilde S},\\
Q'_3{\tilde S}\} $ is a maximal admissible family of ${\tilde I}$ with respect to an admissible order. Note that
$\bigsize({\mathcal G}')=\min\{3,1+2\}=3$ and ${\mathcal G}_1=\{(x_1,x_3,x_4), (x_1,x_2,x_4)\}$ has bigsize $2$. Thus $\bigsize({\mathcal G})=\min\{3,1+2\}=3=1+
\bigsize({\mathcal F})$. We see that $\size_{\tilde S}{\tilde I}=2$ because $Q_1+Q''_2+Q'_3={\tilde \mm}$.
\end{Example}

 \begin{Theorem} \label{t} Let $I$ be a monomial ideal of $S$ and  $I=\cap_{i\in [s]}Q_i$  an irredundant   decomposition  of $I$  as an intersection of irreducible monomial ideals, $P_i=\sqrt{Q_i}$. Then $\sdepth_SI\geq \size_S I+1$.
\end{Theorem}
\begin{proof}  Using \cite[Lemma 3.6]{HVZ} we may reduce to the case when $\sum_{P\in \Ass_S S/I} P=\mm$.
If  $I$ is squarefree then apply Theorem \ref{hpv}, or Theorem \ref{b} with Lemma \ref{l0}. Otherwise, assume that  $c=\deg_nI>1$.  By Lemma \ref{l} there exist $e$ and a monomial ideal ${\tilde I}$ such that  $\sdepth_{{\tilde S}}{\tilde I}\leq\sdepth_SI+1$. Set $I^{(1)}={\tilde I}$ and $S^{(1)}={\tilde S}$. If $I^{(1)}$ is not squarefree then apply again Lemma \ref{l} for some $x_i$ with $\deg_i I^{(1)}>1$. We get  $I^{(2)}=(I^{(1)})^{(1)}$,  $S^{(2)}=(S^{(1)})^{(1)}$ such that $ S^{(2)}=S[x_{n+1},x_{n+2}]$,
   $\sdepth_{ S^{(2)}} I^{(2)}\leq\sdepth_SI+2$.
  Applying Lemma \ref{l} by recurrence we get some monomial ideals $I^{(j)} \subset S^{(j)}$, $j\in [r]$ for some $r$  such that  $S^{(j)}=S[x_{n+1},\ldots,x_{n+j}]$,
  $\sdepth_{ S^{(j)}} I^{(j)}\leq\sdepth_SI+j$ and $I^{(r)}$ is a squarefree monomial ideal (thus $I^{(r)}$ is the polarization of $I$).

Now, let ${\mathcal F}^{(r)}$ be a maximal admissible family of $I^{(r)}$
 with respect to some total admissible order $\nu_r$ such that $\bigsize_{S^{(r)}}(I^{(r)})=\bigsize_{\nu_r}(I^{(r)})=\bigsize({\mathcal F}^{(r)})$.
By Theorem \ref{b} we have  $\sdepth_{S^{(r)}}I^{(r)}\geq 1+ \bigsize({\mathcal F}^{(r)})$.

  Using Lemma \ref{l4} there exists a family ${\mathcal F}^{(r-1)}$  of $I^{(r-1)}$ such that    $1+\bigsize({\mathcal F}^{(r-1)})\leq \bigsize({\mathcal F}^{(r)})$. Applying again Lemma \ref{l4} by recurrence we find a family ${\mathcal F}$  of $I$ such that  $r+\bigsize({\mathcal F})\leq  \bigsize({\mathcal F}^{(r)})$. Thus
$$\sdepth_SI\geq \sdepth_{ S^{(r)}} I^{(r)}-r\geq  $$
$$\bigsize({\mathcal F}^{(r)})-r+1\geq 1+\bigsize({\mathcal F}).$$
 Applying Lemma \ref{l0} we are done.
 \hfill\ \end{proof}
\begin{Remark} Let $n=4$,  $P_1=(x_1,x_2)$, $P_2= (x_1^2,x_3^2)$, $P_3=(x_2,x_4)$, $P_4=(x_3,x_4)$, and $J=\cap_{i\in [4]}P_i$. Note that the polarization of $J$ is the ideal $I$ from Examples \ref{e'}, \ref{e4}, \ref{e5}.
\end{Remark}

\end{document}